\documentclass[letter]{amsart}

\usepackage{amssymb}
\usepackage{amsfonts}
\usepackage{amsmath}
\usepackage{amsthm}
\usepackage{mathtools}
\usepackage{graphicx}
\usepackage{mathrsfs}
\usepackage{dsfont}
\usepackage{amscd}
\usepackage{multirow}
\usepackage[all]{xy}
\normalfont
\usepackage[T1]{fontenc}
\usepackage{calligra}
\usepackage{verbatim}
\usepackage{fourier}
\usepackage[usenames,dvipsnames]{color}
\usepackage[colorlinks=true,linkcolor=Blue,citecolor=Violet]{hyperref}
\usepackage{enumerate}
\usepackage{slashed}
\usepackage{nicematrix}
\usepackage{microtype}
\usepackage{faktor}
\newcommand{\N}{{\mathds{N}}}

\newcommand{\R}{{\mathds{R}}}
\newcommand{\C}{{\mathds{C}}}

\newcommand{\D}{{\mathfrak{D}}}
\newcommand{\A}{{\mathfrak{A}}}
\newcommand{\B}{{\mathfrak{B}}}

\newcommand{\Lip}[1][L]{{\mathsf{#1}}}
\newcommand{\TLip}{{\mathsf{T}}}

\newcommand{\Hilbert}[1][H]{{\mathscr{#1}}}

\newcommand{\dpropinquity}[1]{{\mathsf{\Lambda}^\ast_{#1}}}

\newcommand{\spectralpropinquity}[1]{{\mathsf{\Lambda}^{\mathsf{spec}}_{#1}}}

\newcommand{\Kantorovich}[1]{{\mathsf{mk}_{#1}}}

\newcommand{\Haus}[1]{{\mathsf{Haus}\!\left[{#1}\right]\,}}

\newcommand{\LipschitzD}{{\mathsf{LipD}}}

\newcommand{\StateSpace}{{\mathscr{S}}}

\newcommand{\qcms}{quantum compact metric space}

\newcommand{\sa}[1]{{\mathfrak{sa}\left({#1}\right)}}

\newcommand{\inner}[3]{{\left\langle{#1},{#2}\right\rangle_{#3}}}

\newcommand{\dom}[1]{{\operatorname*{dom}\left({#1}\right)}}

\newcommand{\diam}[2]{{\mathrm{diam}\left({#1},{#2}\right)}}
\newcommand{\qdiam}[2]{{\mathrm{qdiam}\left({#1},{#2}\right)}}

\newcommand{\norm}[2]{\left\|{#1}\right\|_{#2}}

\newcommand{\CDN}{{\mathsf{DN}}}
\newcommand{\TDN}{{\mathsf{TN}}}

\newcommand{\worknote}[1]{}
\newcommand{\opnorm}[3]{{\left|\mkern-1.5mu\left|\mkern-1.5mu\left| {#1} \right|\mkern-1.5mu\right|\mkern-1.5mu\right|_{#3}^{#2}}}

\newcommand{\tunnelextent}[1]{{\chi\left({#1}\right)}}

\newcommand{\tunnelsep}[3]{{\mathsf{sep}_{#3}\left({#2}\middle\vert{#1}\right)}}

\newcommand{\alg}[1]{{\mathfrak{#1}}}

\newcommand{\module}[1]{{\mathscr{#1}}}

\newcommand{\dil}[1]{{\mathrm{dil}\left({#1}\right)}}

\newcommand{\spectrum}[1]{\mathrm{Sp}\left({#1}\right)}

\newcommand{\ModState}[1]{\widehat{\StateSpace}}

\renewcommand{\geq}{\geqslant}
\renewcommand{\leq}{\leqslant}

\newcommand{\Dirac}[1][D]{{\slashed{#1}}}

\theoremstyle{plain}
\newtheorem{theorem}{Theorem}[section]
\newtheorem*{theorem*}{Theorem}

\newtheorem{corollary}[theorem]{Corollary}

\newtheorem{lemma}[theorem]{Lemma}

\newtheorem{theorem-definition}[theorem]{Theorem-Definition}
\newtheorem{hypothesis}[theorem]{Hypothesis}

\theoremstyle{definition}
\newtheorem{definition}[theorem]{Definition}

\newtheorem{notation}[theorem]{Notation}

\theoremstyle{remark}

\newtheorem{remark}[theorem]{Remark}

\numberwithin{equation}{section}

\usepackage[T1]{fontenc}

\begin{document}

\title[]{Continuity for the spectral propinquity of the Dirac operators associated with an analytic path of Riemannian metrics}

\author{Carla Farsi}
\email{carla.farsi@colorado.edu}
\address{Department of Mathematics \\ University of Colorado \\ Boulder CO 80309-0395}

\author{Fr\'{e}d\'{e}ric Latr\'{e}moli\`{e}re}
\address{Department of Mathematics \\ University of Denver \\ Denver CO 80208}
\email{frederic@math.du.edu}
\urladdr{http://www.math.du.edu/\symbol{126}frederic}

\date{\today}

\begin{abstract}
 We prove that a polynomial  path of Riemannian metrics on a closed spin manifold induces a continuous field in the spectral propinquity of metric spectral triples. 
\end{abstract}
\maketitle

%%%%%%%%%%%%%%%%%%%%%%%%%%%%%%%%%%%%%%%%%%%%%%%%%%%%%%%%%%%%%%%%%%%%%%%%%%%%%%%%%%%%%%%%%%%%%%%%%%%%%%%%%
\tableofcontents

\section{Introduction}

The study  of the dependence  of  the eigenvalues of classical  operators such as the Laplacian and the  Dirac on the metric in the setting of a  closed orientable or spin manifold is an important  problem that has seen a lot of recent interest; see  \cite{Berger73, Bando-Urakawa} for some of the earliest work. In this paper we concentrate on the Dirac operator.

Dirac operators are important in both physics when gravity, i.e.,  the space-time metric, is coupled with other interactions,  as well as in  mathematics, where the Dirac operator serves as a tool in Riemannian geometry.
Some of the motivation for studying the dependence of the Dirac operator on the metric comes from Selberg-Witten theory, and paths of metrics  realize  \lq spin geometry in motion.' Calculating the spectrum of the Dirac operator can be a very difficult problem, but in the case of an analytic 
path of metrics, the dependence of the spectrum on the parameter of the path takes the form of a continuous field of eigenvalues and eigenvectors. 
In the seminal paper  \cite{Bo-Gau},  J.-P. Bourguignon and P. Gauduchon, building on work of Y. Kosmann, were the first to  construct a  geometric process to compare spinors for different metrics on a closed spin manifold. (See also \cite{Bar-96} for a different approach that extends beyond the Riemanian case) This made possible the comparison of   Dirac operators associated to different metrics  as they act on these changing  spinor bundles. Importantly, by  collecting the data to be represented on a single Hilbert space  one gets  a holomorphic family of self-adjoint operators of type (A) as in \cite[VII, $\S 2$]{Kato}. 
Fixed two metrics on a closed manifold,  J.-P. Bourguignon and P. Gauduchon in  \cite{Bo-Gau}, defined  an isometry  of associated spinor Hilbert spaces. By using this isometry, it became possible  to transfer the Dirac operators associated to different metrics, which are   defined on different Hilbert spaces, onto the same Hilbert space; 
see also  \cite{Binz83} for precursor's work.
By using  the seminal paper \cite{Bo-Gau}, 
important results on eigenvalues and eigenspaces of the Dirac operator on closed spin manifolds were proven,
\cite{Nowaczyk2013, Nowaczyk2016, NowaczykThesis,Maier}.

As a starting point,  C. B\"ar proved  in \cite[Proposition 7.1]{Bar-96} that a bounded spectral interval of the Dirac operator can be described locally by continuous functions. 
Many global results were inspired by C. B\"ar's result. 
A. Hermann made connections with Kato's perturbation theory by showing  in his thesis \cite[Lemma A.0.12]{Hermann12} 
that an analytic   path of metrics in a closed   spin manifold gives rise to a \emph{holomorphic family of self-adjoint operators of type (A)} \cite[VII,$\S$]{Kato} of \lq translated' Dirac operators all defined on the same Hilbert space and with common domain; see also  \cite{Maier},   \cite{Nowaczyk2013}, \cite{Nowaczyk2016}.
 A special case of this instance is when the analytic  path of metrics is polynomial, or even simpler, a straight-line path cf. \cite{Bo-Gau}, \cite[Proof of Theorem 4.14]{Nowaczyk2016},  \cite[Section 2.3]{Maier},  and \cite[Page 950]{Canzani14}.
 See also\cite{Lesch05}, \cite{Bandara18}, \cite{Ben-Artzi16} for recent results using other metrics, and \cite{Kriegl03} for other types of families of operators.

Several other authors described different aspects of the dependence of the Dirac operator and its eigenvalues and eigenspaces on the metric, as well as variational aspects and interactions with diffeomorphisms groups; see e.g.  \cite{Perez08},  \cite{Ammann16}, \cite{Dabrow86}, \cite{Dabrow13},  \cite{Muller17}, \cite{Karpukhin24}.
	
Spectral triples have emerged in recent times as 
powerful tools to encode geometric data such as classical  operators together with their action  on associated Hilbert spaces. Of particular relevance the  spectral triple associated to a Dirac operator on a closed spin manifold; in this case the elements of the spectral triple are:  1. the C*-algebra of the continuous functions on the manifold; 2. the Hilbert space of the $L^2$-sections of the spinor bundle; and 3.  the Dirac operator.

More in general, for metric spectral triples  (which are spectral triples which induce the weak*-topology on the state space of their C*-algebra) Latr\'emoli\`ere has developed a distance, called \emph{spectral propinquity}  for which distance zero is equivalent to \lq unitarily equivalent' 	\cite{Latremoliere18g, Latremoliere22}. Latr\'emoli\`ere's spectral propinquity was based on the propinquity for quantum compact metric spaces, C*-modules, and many other structures \cite{Latremoliere13b, Latremoliere16b, Latremoliere16c, Latremoliere18b, Latremoliere18c, Latremoliere18d, Latremoliere18g}. 
In addition, when a sequence of metric spectral triples converges to a spectral triple
 in the spectral propinquity, this convergence also implies 
convergence of the bounded functional calculus, and in particular convergence of the eigenvalues, see \cite[Theorem 5.2]{Latremoliere22} for details.

This paper focuses on the study of 
the  dependence of the Dirac operator on the metric
using  Latr\'emoli\`ere's  spectral propinquity framework.
One of the consequences of our main result, Theorem \eqref{main-thm},
is the  following theorem, which provides  a converse of \cite[Theorem 5.2]{Latremoliere22}. This theorem says that  the holomorphic family of self-adjoint operators associated to a polynomial path of metrics  on a closed connected spin manifold give rise to a 
continuous  family in the   spectral propinquity.

\begin{theorem*}(See Theorem \eqref{Riemannian-cv-thm})
	Let $M$ be a closed connected spin manifold. If $t\in I \mapsto g(t)$ is a polynomial path of $C^\infty$ Riemannian metrics over $M$, then 
	$t\in I\rightarrow (C(M),\Gamma^2 \mathrm{Spin}_{g(t)},\Dirac_t)$ is a continuous function for the spectral propinquity.
\end{theorem*}

Continuous fields of \qcms s and  of metric spectral triples depending on a parameter have been considered in several papers in the literature, see e.g.  \cite{Aguilar15, Li05, Kaad21}, 
and our results complement and extend material already available.

As we mentioned, the second author constructed a distance called the \emph{spectral propinquity} on the space of metric spectral triples; we will now review its construction. The spectral propinquity is a distance up to unitary equivalence. Moreover, in appropriate sense, both spectra and bounded continuous functional calculi for the Dirac operators of metric spectral triples are continuous with respect to the spectral propinquity \cite{Latremoliere22}. In this paper, we will see a form of converse when the continuity of the spectrum and of eigenvectors for a family of metric spectral triples over a fixed base implies, in specific cases, the contintuity of that family for the spectral propinquity.

The \emph{spectral propinquity} between two metric spectral triples 
$(\A, \Hilbert_\A, \Dirac_\A)$ and $(\B, \Hilbert_\B, \Dirac_\B)$ is computed in three steps. First, we compute an upper bound for the propinquity between the underlying {\qcms s} $(\A,\opnorm{[\Dirac_\A,\cdot]}{}{\Hilbert_\A})$ and $(\B,\opnorm{[\Dirac_\B,\cdot]}{}{\Hilbert_\B})$. To this end, we define a \emph{tunnel} $\tau=(\D,\Lip_\D,\rho_\A,\rho_\B)$ as a {\qcms} $(\D,\Lip_\D)$, and two quantum isometries $\rho_\A:(\D,\Lip_\D)\rightarrow (\A,\opnorm{[\Dirac_\A,\cdot]}{}{\Hilbert_\A})$, $\rho_\B:(\D,\Lip_\D)\rightarrow(\B,\opnorm{[\Dirac_\B,\cdot]}{}{\Hilbert_\B})$. Given such a tunnel, we define its extend as:
\begin{equation*}
	\tunnelextent{\tau} \coloneqq \max\{ \Haus{\Kantorovich{\Dirac_\A}}(\StateSpace(\D),\rho_\A^\ast(\StateSpace(\A))), \Haus{\Kantorovich{\Dirac_\B}}(\StateSpace(\D),\rho_\B^\ast(\StateSpace(\B))) \} \text.
\end{equation*}
The extent for any such tunnel is an upper bound for the propinquity between $(\A,\opnorm{[\Dirac_\A,\cdot]}{}{\Hilbert_\A})$ and $(\B,\opnorm{[\Dirac_\B,\cdot]}{}{\Hilbert_\B})$, which is indeed defined by:
\begin{multline*}
	\dpropinquity{}((\A,\opnorm{[\Dirac_\A,\cdot]}{}{\Hilbert_\A}),(\B,\opnorm{[\Dirac_\B,\cdot]}{}{\Hilbert_\B})) \coloneqq \\ \inf\left\{ \tunnelextent{\tau} : \tau \text{ tunnel from $(\A,\opnorm{[\Dirac_\A,\cdot]}{}{\Hilbert_\A})$ to $(\B,\opnorm{[\Dirac_\B,\cdot]}{}{\Hilbert_\B}$ }) \right\} \text.
\end{multline*}

Now, as our second step, to account for the actions of the C*-algebras on Hilbert spaces in spectral triples, we restrict ourselves to tunnels which are obtained from diagrams of the form:
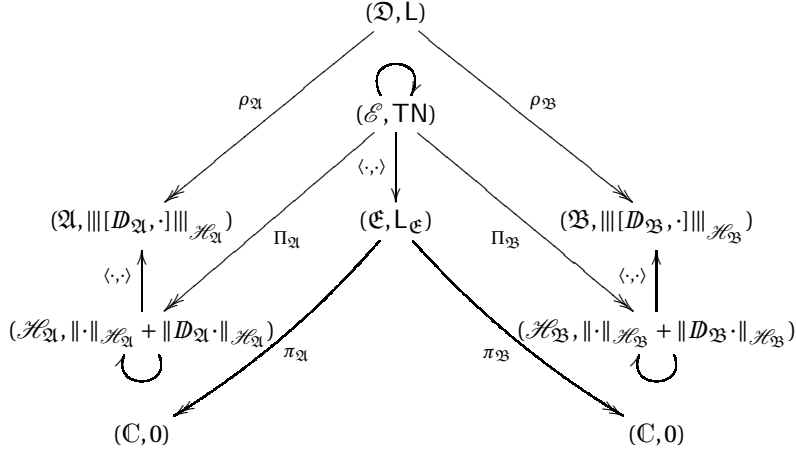
\begin{equation*}
 \xymatrix{
       & (\D,\Lip) \ar@{->>}[ddl]_{\rho_\A} \ar@{->>}[ddr]^{ \rho_\B} & \\
        & (\module{E},\TDN)  \ar@{->>}[ddl]^{\Pi_\A} \ar@{->>}[ddr]_{\Pi_\B} \ar[d]_{\inner{\cdot}{\cdot}{}} \ar@(ul,ur) & \\
        {(\A,\opnorm{[\Dirac_\A,\cdot]}{}{\Hilbert_\A})}  & (\alg{E},\Lip_{\alg{E}}) \ar@{->>}@/^/[ddl]^{\pi_\A} \ar@{->>}@/_/[ddr]_{ \pi_\B} &  {(\B,\opnorm{[\Dirac_\B,\cdot]}{}{\Hilbert_\B})} \\
         (\Hilbert_\A,\norm{\cdot}{\Hilbert_\A} + \norm{\Dirac_\A\cdot}{\Hilbert_\A})\ar[u]^{\inner{\cdot}{\cdot}{}} \ar@(dr,dl) &  & (\Hilbert_\B,\norm{\cdot}{\Hilbert_\B}+\norm{\Dirac_\B\cdot}{\Hilbert_\B}) \ar[u]^{\inner{\cdot}{\cdot}{}} \ar@(dr,dl) & \\
      {(\C,0)}  & & {(\C,0)}  
      }
\end{equation*}
We call such a diagram a \emph{metrical tunnel}. Notably, such a metrical tunnel gives rise to \emph{two} tunnels (the top and the bottom of the diagram): one between $(\A,\opnorm{[\Dirac_\A,\cdot]}{}{\Hilbert_\A})$ and $(\B,\opnorm{[\Dirac_\B,\cdot]}{}{\Hilbert_\B})$, namely $(\D,\Lip_\D,\rho_\A,\rho_\B)$, and also one from $(\C,0)$ to $(\C,0)$ given by $(\alg{E},\Lip_{\alg{E}},\pi_\A, \pi_\B)$. The maximum of the extent of these two tunnels is called the extent for the metrical tunnel. 

Now, our third step is to account for the Dirac operators of the spectral triples in the computation of our distance. This is done by using the group actions induced by these Dirac operators and apply the covariant form of the metrical propinquity. As our actions here are very particular, we can simplify the general presentation to the following. We define, for the metrical tunnel $\tau$ above and any $\varepsilon > 0$:
 	\begin{multline*}
 		\tunnelsep{\tau}{\Dirac_\A,\Dirac_\B}{\varepsilon} \coloneqq \Haus{K_{\varepsilon}} \Big( \Pi_\A^\ast\left\{\xi\in\dom{\Dirac_\A}:\norm{\xi}{\Hilbert_\A}+\norm{\Dirac_\A\xi}{\Hilbert_\A} \leq 1 \right\},  \\ \Pi_\B^\ast\left\{\eta\in\dom{\Dirac_\B}:\norm{\eta}{\Hilbert_\B}+\norm{\Dirac_\B\eta}{\Hilbert_\B}  \leq 1 \right\} \Big)
 	\end{multline*}
where
	\begin{equation*}
 		K_\varepsilon(\xi,\eta) \coloneqq \sup_{\substack{0\leq t \leq \frac{1}{\varepsilon} \\ \TDN(\omega)\leq 1}} \left\{ \left| \inner{\exp( i t \Dirac_\A )\xi}{\omega}{\Hilbert_\A} - \inner{\exp( i t \Dirac_\B )\eta}{\omega}{\Hilbert_\B} \right| \right\}\text.
 	\end{equation*}
  	
 The \emph{spectral propinquity} $\spectralpropinquity{}((\A,\Hilbert_\A,\Dirac_\A),(\B,\Hilbert_\B,\Dirac_\B))$ between two metric spectral triples $(\A,\Hilbert_\A,\Dirac)$ and $(\B,\Hilbert_\B,\Dirac_\B)$ is:
 		\begin{equation*}
 			\inf\Big\{ \varepsilon > 0 : \exists \tau \text{ tunnel from }(\A,\Hilbert_\A,\Dirac_\A)  \text{ to }(\B,\Hilbert_\B,\Dirac_\B) \,
 			\hbox{ with }\, \max\{\tunnelextent{\tau}, \tunnelsep{\tau}{\Dirac_\A,\Dirac_\B}{\varepsilon} \} < \varepsilon \Big\}\text.
 		\end{equation*}
 
The  spectral propinquity $\spectralpropinquity{}$ is a metric up to unitary equivalence on the space of metric spectral triples., and the spectrum of the Dirac operators and the continuous functional calculus are in some sense continuous with respect to the spectral propinquity, see  \cite{Latremoliere22}.

\section*{Acknowledgments} This research was partly supported by the  Simon Foundation Collaboration Grant MPS-TSM-00007731 (C.F.). 
'

\section{The main results}

This paper was motivated  by the following    natural geometric question: on a connected closed spin manifold   how can one   reframe the dependence of the Dirac operator 
on the metric by using the spectral propinquity.
We actually address this question in the more general context of fields of spectral triples and associated eigenvalues/eigenvectors. 
Our results provide a very partial converse to some of the results in \cite{Latremoliere22}, in the sense that we use  the continuity of spectra to obtain convergence for the propinquity. This  in turn implies  additional results on the  continuous functional calculus, see \cite[Theorem 4.7, Corollary 4.8, Theorem 4.9]{Latremoliere22}.

Our main result is the  following theorem, which provides  a converse of \cite[Theorem 5.2]{Latremoliere22}. This theorem says that  continuous families of self-adjoint operators  satisfying certain hypotheses give rise to 
continuous  families in the   spectral propinquity.

\begin{theorem}\label{main-thm}
	Let $\A$ be a unital separable C*-algebra acting on a Hilbert space $\Hilbert$. Assume that for each $t\in [0,1]$, we are given a metric spectral triple $(\A,\Hilbert,\Dirac_t)$ such that the following properties hold:
	\begin{enumerate}
		\item for all $\varepsilon > 0$, there exists $\delta>0$ such that, if $t\in[0,\delta)$, there exists a tunnel  from  $(\A,\Hilbert,\Dirac_0)$ to $(\A,\Hilbert,\Dirac_t)$ of the form $\tau_{\varepsilon,t} \coloneqq (\A\oplus\A,\TLip,j_1,j_2)$, where $j_1$ and $j_2$ are the canonical surjections on the first and second summands respectively, and
		\begin{equation*}
			\TLip(a,b)\coloneqq \max\left\{ \Lip_0(a), \Lip_t(b), \frac{2}{\varepsilon}\norm{a-b}{\A} \right\}\text;
		\end{equation*}
		\item there exist a sequence $(\lambda_n)_{n\in\N}$ of continuous functions from $[0,1]$ to $\R$, and a sequence $(e_n)_{n\in\N}$ of continuous functions form $[0,1]$ to $\Hilbert$, such that:
		\begin{enumerate}
			\item $(e_n(t))_{n\in\N}$ is a Hilbert basis of $\Hilbert$ for all $t\in[0,1]$,
			\item $\Dirac_t e_n(t) = \lambda_n(t) e_n(t)$ for all $n\in\N$, $t\in[0,1]$,
			\item for all $\Lambda>0$, there exists $\delta>0$ such that, for all $t\in [0,\delta)$, we have \begin{equation*}
			|\spectrum{\Dirac_t}\cap[-\Lambda,\Lambda]| = |\spectrum{\Dirac_0}\cap[-\Lambda,\Lambda]|\text.
			\end{equation*}
		\end{enumerate} 
	\end{enumerate}
	Then the spectral triples $(\A,\Hilbert,\Dirac_t)_{t\in [0,1]} $ converge to $(\A,\Hilbert,\Dirac_0)$ as $t\to 0$ in the spectral propinquity:
	\begin{equation*}
		\lim_{t\rightarrow 0} \spectralpropinquity{}((\A,\Hilbert,\Dirac_t),(\A,\Hilbert,\Dirac_0)) = 0 \text.
	\end{equation*}
\end{theorem}	
We will prove Theorem \eqref{main-thm} in Section \eqref{sec:lemmas-and-proof}.

After we establish Theorem \eqref{main-thm}, we turn our attention 
to the situation in which the  family of metric spectral triples is associated 
to the variation of the Riemanian metric  along a polynomial path, see   our second main result  Theorem \eqref{Riemannian-cv-thm}. This theorem  provides an answer to our motivating question.

We now recall a few definitions and results to introduce the required notation to state our second main theorem.

As we will work with families of metrics, the various ``indices'' and parentheses involved in standard notations for Riemannian metrics and vector fields tend to become hard to read, so we shall adopt a useful variation, directly taken from the usual construction of the Hilbert module of tangent vector fields over a Riemannian manifold.

\begin{notation}
	If $M$ is a $C^k$ differentiable manifold, then we denote by $T^{p,q}M \coloneqq TM^{\otimes p} \otimes (T^\ast M)^\otimes q$ be the bundle of $(p,q)$ $C^k$-tensors over $M$.
\end{notation}

\begin{definition}\label{poly-path-def}
	A \emph{polynomial path of $C^k$-Riemannian metrics} $t\in [0,1] \mapsto g(t)$ is a function from $[0,1]$ to the set of all $C^k$-Riemannian metrics over $M$ for which there exist $h_0,\ldots,h_N \in \Gamma(Sym^{0,2}M)$ such that:
	\begin{equation*}
		g(t) = \sum_{j=0}^N t^j h_j \text.
	\end{equation*}
\end{definition}

	Let $t\in I\coloneqq [0,1] \mapsto g(t)$ be a polynomial path of $C^\infty$ Riemannian metrics over $M$. For each $t\in I$, let $\Gamma^2 \mathrm{Spin}_{g(t)}$ be the Hilbert space of square integrable sections of the spinor bundle over $M$ for the metric $g(t)$, and $\Dirac^t$ the associated Dirac operator. We also denote $\Gamma^2 \mathrm{Spin}_{g(0)}$ by $\Hilbert$, and by $\Dirac_0$ by  $\Dirac$. 
	
Since polynomial paths of $C^\infty$-Riemannian metrics are, in particular, analytic paths of metrics, by \cite{Bo-Gau, NowaczykThesis, Maier, Hermann12}, there exists a family of unitaries $t\in [0,1] \mapsto \beta(t)$ with $\beta(t):\Gamma^2 \mathrm{Spin}_{g(t)} \rightarrow \Hilbert$, such that:
\begin{itemize}
	\item $\beta(t)$ is a unitary from $\Gamma^2 \mathrm{Spin}_{g(t)}$ onto $\Hilbert$, which intertwines the action of $C(M)$ on $\Gamma^2 \mathrm{Spin}_{g(t)}$ and $\Hilbert$ (note that we will omit writing a special symbol for these representations),
	\item If we set, $\Dirac_t \coloneqq \beta(t)\Dirac^t \beta(t)^\ast$, for all $t\in [0,1]$, then $t\in [0,1] \mapsto \Dirac_t$ is a holomorphic family of self-adjoint operators of type (A) \cite[Section VII \S 2]{Kato}.
\end{itemize}

We are now ready to state our second main result, Theorem 
\eqref{Riemannian-cv-thm}, which will be derived in Section \eqref{sec:pf-sec-main-thm} from Theorem \eqref{main-thm}.

\begin{theorem}\label{Riemannian-cv-thm}
	Let $M$ be a closed connected spin manifold. If $t\in I \mapsto g(t)$ is a polynomial path of $C^\infty$ Riemannian metrics over $M$, then $t\in I\rightarrow (C(M),\Gamma^2 \mathrm{Spin}_{g(t)},\Dirac_{t})$ is a continuous function for the spectral propinquity.
\end{theorem}

\section{Families of spectral triples and proof of Theorem \eqref{main-thm}}\label{sec:lemmas-and-proof}

In this section we will prove Theorem \eqref{main-thm}; the following 
are  hypotheses which we  will assume throughout this section.

\begin{hypothesis}\label{type-A-hyp}
	Assume that  $(\A,\Hilbert,\Dirac_t)_{t\in [0,1]}$ is a family of metric spectral triples for which there exists a family $(\alpha_n)_{n\in\N}$ of continuous $\R$-valued functions over $[0,1]$, and a family $(e_n)_{n\in\N}$ of functions over $[0,1]$, valued in $\Hilbert$, such that:
	\begin{enumerate}
		\item $(e_n(t))_{n\in\N}$ is a Hilbert basis of $\Hilbert$,
		\item For any fixed  $t\in [0,1]$,  $\lim_{n\rightarrow\infty} \alpha_n(t) = \infty$, 
		\item $(\alpha_n(0))_{n\in\N}$ is weakly increasing,
		\item $\Dirac^2(t) e_n(t) = \alpha_n(t) e_n(t)$ for all $t\in [0,1]$ and $n\in\N$,
		\item for all $\Lambda>0$, there exists $\delta>0$ such that, for all $t\in [0,\delta)$, we have \begin{equation*}
			|\spectrum{\Dirac_t^2}\cap[0,\Lambda]| = |\spectrum{\Dirac_0^2}\cap[0,\Lambda]|\text,
		\end{equation*}
	\end{enumerate}
	The graph norm of $\Dirac_t$ on its domain $\dom{\Dirac_t}$ is denoted by $\CDN_t$.
\end{hypothesis}

We now detail  a  succession of lemmas which relate various continuity properties of metric spectral triples to properties of their D-norms and domains. 
These lemmas will be used in the proof of Theorem \eqref{main-thm}.
Our main lemma, Lemma  \eqref{uniform-truncation-lemma} below, establishes a form of uniform truncation of vectors of controlled D-norms under our Hypothesis \eqref{type-A-hyp}.

\begin{lemma}\label{constant-N-lemma}
	If we assume Hypothesis (\ref{type-A-hyp}), then for all $\Lambda>0$ such that $\Lambda\notin\{\alpha_n(0):n\in\N\}$, there exists $\delta>0$ and $N\in\N$ such that, if $t\in[0,\delta]$, and if $n\geq N$, then $\alpha_n(t) > \Lambda$\text.
\end{lemma}

\begin{proof}
	Let $N\in\N$ be given by $\{\alpha_0(0),\ldots,\alpha_N(0)\} = [0,\Lambda]\cap\{\alpha_n(0):n \in\N\}$. Let $\varepsilon \coloneqq \frac{1}{2}(\Lambda - \max_{n=0}^N \alpha_n(0))$; note that by assumption, $\varepsilon > 0$. By continuity, there exists $\delta_1 > 0$ such that $|\alpha_n(t) - \alpha_n(0)| < \varepsilon$ for all $t\in [0,\delta_1)$, and $j\in\{0,\ldots,N\}$. Therefore, for all $t\in[0,\delta_1)$, we have $\{\alpha_n(t):n\in\{0,\ldots,N\}\} \subseteq [0,\Lambda]$. 
	By Assumption (5), there also exists $\delta>0$ such that $|\{\alpha_n(t):n\in\N\}\}| = N+1$ for all $t\in[0,\delta)$. Therefore, $[0,\Lambda)\cap\{\alpha_n(t):n\in\N\}=\{\alpha_n(t):n\leq N\}$ for all $t\in [0,\min\{\delta_1,\delta\}]$.
\end{proof}

\begin{lemma}\label{uniform-truncation-lemma}
	If we assume Hypothesis (\ref{type-A-hyp}), then, for all $\varepsilon > 0$, there exists $N\in\N$ and $\delta>0$ such that, for all  $t\in [0,\delta]$, and for all $\xi \in \dom{\Dirac_0}$,
	\begin{equation}\label{eq:proj-onto-trunc}
		\norm{\xi - P_N(t)\xi}{\Hilbert} \leq \varepsilon\, \CDN_t(\xi) \text,
	\end{equation}
	where  $P_N(t)$ is the orthogonal projection onto $\mathrm{span}\{e_1(t),\ldots,e_N(t)\}$.
\end{lemma}

\begin{proof}
	We define $\mu_n(t) \coloneqq \alpha_n(t) + 1$ for all $n\in\N$ and $t\in [0,1]$. We note that since $\alpha_n(t)\geq 0$ by construction, $\mu_n(t) \geq 1 > 0$ for all $t \in [0,1]$ and $n\in\N$.
	
	Let $\varepsilon > 0$ and let $\Lambda = \frac{8}{\varepsilon^2} >0$. By Lemma (\ref{constant-N-lemma}), there exists $N\in\N$ and $\delta_0>0$ such that, for all $t \in [0,\delta_0)$, and for all $n\geq N$, we have $\mu_n(t) > \alpha_n(t) > \Lambda$.
	
	Let $\delta\coloneqq\min\{\delta_0,\delta_1\}$, and fix $t \in [0,\delta)$. Assume that  $\xi \in \dom{\Dirac_t}$ with $\CDN_t(\xi) \leq 1$.
	
	We write $\xi = \sum_{n\in\N} a_n(t) e_n(t)$ for $(a_n(t))_{n\in\N} \in \ell^2(\N)$ with $a_n(t) = \inner{\xi}{e_n(t)}{\Hilbert}$ for all $n\in\N$. Since $e_n$ is continuous, so is $a_n$. 
	
	As $(e_n)_{n\in\N}$ is a Hilbert basis and $\norm{\xi}{\Hilbert}\leq \CDN_t(\xi)\leq 1$, we have $\sum_{n\in\N} |a_n(t)|^2 \leq 1$. Moreover, 
	\begin{equation*}
		\norm{\Dirac_t\xi}{\Hilbert}^2 = \inner{\Dirac_t\xi}{\Dirac_t\xi}{\Hilbert} = \inner{\Dirac_t^2\xi}{\xi}{\Hilbert} = \sum_{n\in\N} \alpha_n(t) |a_n(t)|^2 \text. \end{equation*}

	Also, note that since $\CDN_t(\xi)\leq 1$, we have:
	\begin{equation*}
		\sum_{n\in\N} \mu_n(t) |a_n|^2 = \sum_{n\in\N} |a_n(t)|^2 + \sum_{n\in\N} \alpha_n(t) |a_n(t)|^2 = \norm{\xi}{\Hilbert}^2 + \norm{\Dirac_t\xi}{\Hilbert}^2 \leq 2 \CDN_t(\xi)^2 \leq 2\text{.}
	\end{equation*}
	
	We then have:
	\begin{align*}
		\sum_{n\geq N} |a_n(t)|^2 &= \sum_{n\geq N} \frac{|a_n(t)|}{\sqrt{\mu_n(t)}} \sqrt{\mu_n(t)} |a_n(t)|\\
		&\leq \sqrt{\sum_{n\geq N} \frac{|a_n(t)|^2}{\mu_n(t)}} \sqrt{\sum_{n\geq N} \mu_n(t) |a_n(t)|^2}\\
		&\leq \sqrt{\sum_{n\geq N} \frac{|a_n(t)|^2}{\mu_n(t)}} \cdot 2 \text{.}
	\end{align*}
	
	Using Abel summation and since $\sum_{n\in\N} |a_n(t)|^2\leq 1$, for all $n > N$:
	\begin{align*}
		\sum_{j = N}^n \frac{|a_j(t)|^2}{\mu_j(t)} 
		&= \frac{1}{\mu_{n+1}(t)}\sum_{j=0}^n |a_j(t)|^2 - \frac{1}{\mu_N(t)}\sum_{j=0}^N |a_j(t)|^2 \\
		&\quad + \sum_{j=N}^n \left(\sum_{m=0}^j |a_m(t)|^2\right) \left( \frac{1}{\mu_{j+1}(t)}-\frac{1}{\mu_{j}(t)} \right)\\
		&\leq \frac{1}{\mu_{n+1}(t)} + \sum_{j=N}^n \left(\frac{1}{\mu_{j+1}(t)}-\frac{1}{\mu_j(t)}\right)\\
		&= \frac{2}{\mu_{n+1}(t)} - \frac{1}{\mu_N(t)} \leq \frac{2}{\mu_{n+1}(t)} \\
		&\leq \frac{\varepsilon^2}{4} \text{.}
	\end{align*}
	Hence $\norm{\xi-P_N(t)\xi}{\Hilbert} \leq 2 \frac{\varepsilon}{2} = \varepsilon$. By homogeneity, we conclude that for all $\xi \in \dom{\Dirac_t}$, we have $\norm{\xi-P_N(t)}{\Hilbert} \leq \varepsilon \, \CDN_t(\xi)$.  
\end{proof}

Now, we prove that if we restrict to some finite dimensional subspaces of the common domain. under our hypothesis we have a continuous field of D-norms  

\begin{lemma}\label{continuous-field-D-norms-lemma}
	If we assume Hypothesis (\ref{type-A-hyp}), then for any  $C>0$ and  $N\in\N$, the family $(\norm{\cdot}{t})_{t\in [0,1]}$ of norms on $\C^{N+1}$, defined for each $t\in [0,1]$ by
	\begin{equation*}
		\norm{\cdot}{t,N} : (z_0,\ldots,z_N) \in \C^{N+1} \mapsto \CDN_t \left(\sum_{n=0}^ N z_n e_n(t)\right)
	\end{equation*}
	converges uniformly to $\norm{\cdot}{0}$ on the closed ball of radius $C$, center $0$, in $\C^{N+1}$.
\end{lemma}

\begin{proof}
	Since $\alpha_0,\ldots,\alpha_N$ are continuous over the compact $[0,1]$, they are bounded. Let
	\begin{equation*}
		M \coloneqq \sup \left\{ \alpha_j(t) : t \in [0,1], j \in \{ 0,\ldots,N \} \right\} \text.
	\end{equation*}
	
	Denote by $\norm{\cdot}{\C^{N+1}}$ the usual $2$-norm on $\C^{N+1}$. Fix $t\in [0,1]$. Let $(z_0,\ldots,z_N) \in \C^{N+1}$. We note that:
	\begin{align}\label{dnormfield-eq-0}
		\norm{(z_0,\ldots,z_N)}{t,N} 
		&= \norm{(z_0,\ldots,z_N)}{\C^{N+1}} + \sqrt{\sum_{j=0}^N |\alpha_j(t)| |z_j|^2} \\
		&\leq (1+\sqrt{M}) \norm{(z_0,\ldots,z_N)}{\C^{N+1}}\text. \nonumber
	\end{align}
	Therefore, for all $z,z' \in \C^{N+1}$, 
	\begin{equation*}
		\left| \norm{z}{t,N} - \norm{z'}{t,N} \right| \leq \norm{z-z'}{t,N} \leq (1+\sqrt{M}) \norm{z-z'}{\C^{N+1}}\text.
	\end{equation*}
	Therefore, $(\norm{\cdot}{t,N})_{t\in [0,1]}$ is an equicontinuous family of continuous functions. Moreover, Expression (\ref{dnormfield-eq-0}) also shows that, since $\alpha_0$,\ldots,$\alpha_N$ are continuous, $t\in[0,1]\mapsto\norm{\cdot}{t}$ converges pointwise to $\norm{\cdot}{0}$ as $t\rightarrow 0$. In particular, by Arz{\'e}la-Ascoli, $t\mapsto\norm{\cdot}{t,N}$ converges uniformly to $\norm{\cdot}{0}$ over the compact set $\{z\in\C^{N+1} :  \norm{z}{\C^{N+1}}\leq C \}$ for any $C>0$. 
\end{proof}

We now can bring together all of the above  lemmas to establish Theorem (\ref{main-thm}).

\begin{proof}[{Proof of Theorem (\ref{main-thm})}]
	Once more, we denote the graph norm of $\Dirac_t$ on its domain $\dom{\Dirac_t}$ by $\CDN_t$.
	
	Let $\varepsilon > 0$, and let $\alpha_n(t) = \lambda_n(t)^2$; we may re-index $\lambda_n$ so that $(\alpha_n(0))_{n\in\N}$ is weakly increasing. By assumption, we now meet the assumptions of Hypothesis (\ref{type-A-hyp}).
	
	By Lemma (\ref{uniform-truncation-lemma}), there exists $\delta_0> 0$ and $N \in \N$ such that, for all $t\in[0,\delta]$, if $\xi \in \dom{\Dirac_0}$ with $\CDN_t(\xi)\leq 1$, then:
	\begin{equation}\label{eq:trunc-ep-6}
		\norm{\xi - P_N(t)\xi}{\Hilbert} < \frac{\varepsilon}{6}\CDN_t(\xi) \text.
	\end{equation}
	
	By Lemma (\ref{continuous-field-D-norms-lemma}), there exists $\delta_1>0$ such that
	\begin{equation*}
		\sup_{\substack{ (z_0,\ldots,z_N)\in\C^{N+1} \\ \norm{(z_0,\ldots,z_N)}{\C^{N+1}} \leq 1} } |\norm{(z_0,\ldots,z_N)}{t} - \norm{(z_0,\ldots,z_N)}{0}| < \frac{\varepsilon}{6} \text.
	\end{equation*}
	In particular, if $\xi \in \dom{\Dirac_0}$, if $t \in [0,\min\{\delta_0,\delta_1\}]$, if $\{s,r\} = \{0,t\}$, and if we write $\xi = \sum_{j=0}^N z_j e_j(s)$, then by homogeneity,
	\begin{equation}\label{main-thm-eq-2}
		\CDN_r(\sum_{j=0}^N z_j e_j(r)) \leq \CDN_s(\xi) \frac{6+\varepsilon}{6}\text,\end{equation}
	since $\CDN_t(\sum_{j=0}^N z_j e_j(r))\leq 1$ implies $\norm{\sum_{j=0}^N z_j e_j(r)}{\Hilbert}\leq 1$, i.e. $\norm{(z_0,\ldots,z_N)}{\C^{N+1}} \leq 1$.
	
	By assumption, let $\delta_2> 0$ such that for all $t \in [0,\delta_2]$, there exists a tunnel $\tau_t$ from $(\A,\opnorm{[\Dirac_0,\cdot]}{}{\Hilbert})$ to $(\A,\opnorm{[\Dirac_t,\cdot]}{}{\Hilbert}$ of the form given in our assumption. Note that a standard calculation \cite{Latremoliere13b} shows that the extent of $\tau_{t}$ is at most $\varepsilon$.
	
	By continuity, there exists $\delta_3 > 0$ such that 
	\begin{equation}\label{main-thm-eq-3}
		\norm{e_n(t)-e_0(t)}{\Hilbert} < \frac{\varepsilon}{12(N+1)}
	\end{equation}
	for all $n \in \{0,\ldots,N\}$ and $t\in[0,\delta_1]$.
	
	By continuity, there also exists $\delta_4>0$ such that for all $x,t \in [0,\delta_4]$:
	\begin{equation}
		\sup\{ |\exp(i x \lambda_j(t)) - \exp(i x \lambda_j(0))| : j \in \{0,\ldots,N\} \} < \frac{\varepsilon}{12}\text.
	\end{equation}

	Let $\delta \coloneqq \min\{\delta_0,\delta_1,\delta_2, \delta_3, \delta_4 \} > 0$.
	
	Fix  $t \in [0,\delta]$. We define, for all $\xi,\eta \in \Hilbert$:
	\begin{equation*}
		\TDN_t(\xi,\eta) \coloneqq \max\left\{ \CDN_0(\xi), \CDN_t(\eta), \frac{2}{\varepsilon}\norm{\xi-\eta}{\Hilbert} \right\} \text,
	\end{equation*}
	allowing for the value $\infty$.
	
	Let $\xi\in\dom{\Dirac_0}$ with $\CDN_0(\xi) = 1$, and let 
	\begin{equation*}
		\eta \coloneqq \frac{6}{6+\varepsilon}\sum_{n\leq N} \inner{\xi}{e_n(0)}{\Hilbert} e_n(t)\text.
	\end{equation*}
	We record that $\CDN_t(\eta)\leq 1$ by Expression \eqref{main-thm-eq-2}. Moreover, 
	\begin{align*}
		\norm{\xi-\eta}{\Hilbert}
		&\leq \underbracket[1pt]{\norm{\xi - P_N(0)\xi}{\Hilbert}}_{\leq\frac{\varepsilon}{6} \text{ by Exp. \eqref{eq:trunc-ep-6}}} + \norm{P_N(0)\xi - \sum_{n\leq N} \inner{\xi}{e_n(0)}{\Hilbert} e_n(t)}{\Hilbert} \\
		&\quad + \norm{\sum_{n\leq N} \inner{\xi}{e_n(0)}{\Hilbert} e_n(t)-\eta}{\Hilbert} \\
		&\leq \frac{\varepsilon}{6} + \sum_{n\leq N} \underbracket[1pt]{\norm{e_n(0)-e_n(t)}{\Hilbert}}_{\leq \frac{\varepsilon}{12(N+1)} \text{ by Exp. \eqref{main-thm-eq-3}}} + \frac{\varepsilon}{6+\varepsilon} \\
		&\leq \frac{\varepsilon}{6} + \frac{\varepsilon}{12} + \frac{\varepsilon}{6} \leq \frac{\varepsilon}{2} \text.
	\end{align*}
	Hence $\TDN(\xi,\eta) = 1$.
	
	A similar computation can be made for $\xi \in \Hilbert$ with $\CDN_t(\xi) = 1$, in which case we set  $\eta \coloneqq  \frac{6}{6+\varepsilon}\sum_{n\leq N} \inner{\xi}{e_n(t)}{\Hilbert} e_n(0)$; as above, $\TDN(\eta,\xi) = 1$.
	
	Of course, $\Hilbert\oplus\Hilbert$ is an $\A\oplus\A$-module for the diagonal action $(a,b)(\xi,\eta) = (a\xi,b\eta)$ for all $a,b\in\A$ and $\xi,\eta\in\Hilbert$.
	
	We now check the necessary Leibniz conditions.
	\begin{align*}
		\frac{2}{\varepsilon}\norm{a\xi-b\eta}{\Hilbert}
		&\leq \frac{2}{\varepsilon}\left( \norm{a(\xi-\eta)}{\Hilbert} + \norm{(a-b)\eta}{\Hilbert} \right) \\
		&\leq \norm{a}{\A}\left(\frac{2}{\varepsilon}\underbracket[1pt]{\norm{\xi-\eta}{\Hilbert}}_{\leq \TDN(\xi,\eta)} \right)+ \left(\underbracket[1pt]{\frac{2}{\varepsilon}\norm{a-b}{\A}}_{\leq \TLip(a,b)}\right)\norm{\eta}{\Hilbert} \\
		&\leq \left(\norm{(a,b)}{\A\oplus\B} + \TLip(a,b))\right)\TDN(\xi,\eta) \text.
	\end{align*}
	
	Our secondary tunnel is given by considering $\Hilbert\oplus\Hilbert$ as a $\C^2$ Hilbert module, with $\inner{(\xi,\eta)}{(\xi',\eta')}{\C^2} \coloneqq \left(\inner{\xi}{\xi'}{\Hilbert}, \inner{\eta}{\eta'}{\Hilbert}\right)$, and of course, $(\xi,\eta)(\lambda,\mu) = (\lambda\xi,\mu\eta)$, for all $\xi,\xi',\eta,\eta'\in\Hilbert$ and $\lambda,\mu \in \C$. If $j_1:(z,w)\in\C^2\mapsto z$ and $j_2(z,w)\in\C^2\mapsto w$, then we consider our secondary tunnel as $(\C^2,\Lip[Q],j_1,j_2)$ with
	\begin{equation*}
		\Lip[Q](z,w) \coloneqq \frac{2}{\varepsilon} |z-w| \text.
	\end{equation*}
	
	The inner Leibniz condition holds:
	\begin{align*}
		\frac{2}{\varepsilon} |\inner{\xi}{\xi'}{\Hilbert} - \inner{\eta}{\eta'}{\Hilbert}|
		&\leq \frac{2}{\varepsilon} |\inner{\xi-\eta}{\xi'}{\Hilbert} - \inner{\eta}{\eta'-\xi'}{\Hilbert}| \\
		&\leq \frac{2}{\varepsilon}\norm{\xi-\eta}{\Hilbert} \norm{\xi'}{\Hilbert} + \norm{\eta}{\Hilbert} \frac{1}{\varepsilon}\norm{\xi'-\eta'}{\Hilbert} \\
		&\leq \TDN(\xi,\eta) \norm{\xi'}{\Hilbert} + \norm{\eta}{\Hilbert}\TDN(\xi',\eta') \\
		&\leq \TDN(\xi,\eta) \TDN(\xi',\eta') + \TDN(\xi,\eta)\TDN(\xi',\eta') \\
		&= 2  \TDN(\xi,\eta) \TDN(\xi',\eta')\text.
	\end{align*}
	The extent of the tunnel $(\C^2,\Lip[Q])$ is $\frac{\varepsilon}{2}$ by construction.
	
	Our metrical tunnel is thus given by the metrical C*-correspondence $(\Hilbert\oplus\Hilbert,\TDN,\A\oplus\A,\TLip,\C^2,\Lip[Q])$, together with the quantum isometries $(\Pi_1,\pi_1,j_1)$ and $(\Pi_2,\pi_2,j_2)$, where:
	\begin{equation*}
		\Pi_1 : (\xi,\eta)\in\Hilbert\oplus\Hilbert\mapsto \xi\text{, } \Pi_2:(\xi,\eta)\in\Hilbert\oplus\Hilbert\mapsto\eta\text{, }\pi_1:(a,b)\in\A\oplus\A\mapsto a\text{ and }\pi_2:(a,b)\mapsto b\text.
	\end{equation*}
	We proved the extend of this tunnel is $\frac{\varepsilon}{2}$.
	
	Last, we check the covariant reach for our tunnel, and for the actions of the monoid $[0,\infty)$ given by exponentiating the Dirac operators. let $\xi \in\dom{\Dirac_t}$ with $\CDN_s(\xi)\leq 1$. As above, let $\eta\coloneqq \frac{6}{6+\varepsilon}\sum_{j=0}^N z_j  e_j(0)$ where $z_j \coloneqq \inner{\xi}{e_j(t)}{\Hilbert}$ for each $j\in\{0,\ldots,N\}$. Once again, by construction, $\CDN_0(\eta)\leq 1$. We then have:
	\begin{align*}
		\norm{\exp(i x \Dirac_t)\xi - \exp(it\Dirac_0) \eta}{\Hilbert}
		&\leq \norm{\exp(i x \Dirac_t)\xi - \exp(it\Dirac_0) \sum_{j=0}^N  z_j e_j(0)}{\Hilbert} \\&\quad+ \norm{\exp(i x \Dirac_0)\sum_{j=0}^N z_j e_j(0) - \exp(it\Dirac_0) \eta}{\Hilbert} \\
		&\leq \norm{\exp(i x \Dirac_t)(\xi - P_N(t)\xi)}{\Hilbert} \\
		&\quad+ \norm{\exp(i x \Dirac_t) P_N(t)\xi - \exp(i x \Dirac_0)\sum_{j=0}^N z_j e_j(0))}{\Hilbert} + \frac{\varepsilon}{6} \\
		&\leq \frac{\varepsilon}{3} + \norm{\sum_{j=0}^N z_j(\exp(ix\lambda_j(t))-\exp(i x\lambda_j(0)))e_j(t)}{\Hilbert} \\ 
		&\quad+ \norm{\sum_{j=0}^N z_j \exp(i x \lambda_j(0)) (e_j(t)-e_j(0))}{\Hilbert} +  \frac{\varepsilon}{6}\\
		&\leq \frac{\varepsilon}{6} + \sum_{j=0}^N |1-\exp(ix (\lambda_j(t)-\lambda_j(0)))| + \sum_{j=0}^N \norm{e_j(t)-e_j(0)}{\Hilbert} + \frac{\varepsilon}{3}\\
		&\leq \frac{\varepsilon}{6} + \frac{\varepsilon}{12} + \frac{\varepsilon}{12} + \frac{\varepsilon}{6} = \frac{\varepsilon}{2} \text.
	\end{align*} 
	
	Therefore, for all $(\omega,\omega') \in \dom{\TDN}$ with $\TDN(\omega',\omega) \leq 1$ (so, in particular, $\CDN_t(\omega)\leq 1$ and  $\norm{\omega}{\Hilbert}\leq 1$),
	\begin{align*}
		\left|\inner{\exp(i x \Dirac_t)\xi}{\omega}{\Hilbert} - \inner{\exp(i x \Dirac_0)\eta}{\omega'}{\Hilbert}\right| 
		&\leq \left|\inner{\exp(i x \Dirac_t)\xi-\exp(i x \Dirac_0)\eta}{\omega}{\Hilbert}\right| \\
		&\quad + \left|\inner{\exp(i x \Dirac_0)\eta}{\underbracket[1pt]{\omega-\omega'}_{\TDN(\omega,\omega')\leq 1\implies\norm{\omega-\omega'}\leq\frac{\varepsilon}{2}}}{\Hilbert}\right| \\
		&\leq \norm{\exp(i x \Dirac_t)\xi - \exp(it\Dirac_0) \eta}{\Hilbert} + \norm{\omega-\omega'}{\Hilbert} \\
		&\leq \frac{\varepsilon}{2} + \frac{\varepsilon}{2} = \varepsilon \text.
	\end{align*}
	
	The same computation applies if $\CDN_0(\xi)\leq 1$. Therefore:
	\begin{equation*}
		\spectralpropinquity{}((\A,\Hilbert,\Dirac_0),(\A,\Hilbert,\Dirac_t))\leq \varepsilon \text,
	\end{equation*}
	and so the proof of Theorem \eqref{main-thm} is finished.
\end{proof}

\section{Lipschitz convergence and proof of Theorem \eqref{Riemannian-cv-thm}}
\label{sec:pf-sec-main-thm}

Polynomial paths of $C^k$-Riemannian metrics have a natural Lipschitz property, which will give us in turn, Lipschitz convergence of their underlying metric spaces.
We start with fixing some notation.
\begin{notation}
	Let  $M$ be a closed connected Riemannian  manifold, and  $g$ be  a Riemannian metric on it,  that is, an element os  $\Gamma^\infty T^{0,2}M$ . If  $x \in M$, and $X,Y \in T_x M$ are two tangent vectors at $x$, then we denote $g_x(X,Y)$ by $\inner{X}{Y}{g,x} \in \R$ (indeed the metric at  $x$ is some bilinear form on $T_x M$, to which we can then apply to $X$ and $Y$). This notation  will avoid confusion when working with fields of metrics.
\end{notation}

\begin{lemma}\label{diff-poly-path-lemma}
	If $g$ is a polynomial path of metrics, then there exists $C>0$ such that, for all $t,t_0 \in [0,1]$, and for all $(x,X) \in TM$:
	\begin{equation*}
		\left| \inner{X}{X}{g(t),x} - \inner{X}{X}{g(t_0),x} \right| \leq C|t-t_0| \inner{X}{X}{g(0),x} \text.
	\end{equation*}
\end{lemma}

\begin{proof}
	A simple computation shows that since, for all $x\in M$ and $X,Y \in T_x M$,
	\begin{equation*}
		\inner{X}{Y}{g(t),x} = \sum_{j=0}^N t^j \inner{X}{Y}{h_j, x}\text,
	\end{equation*}
	we have
	\begin{equation*}
		\frac{d}{dt}\inner{X}{Y}{g(t),x} = \sum_{j=0}^{N-1} (j+1) t^j \inner{X}{Y}{h_j,x}\text.
	\end{equation*}
	We now prove that this latter expression defines a bounded function over $TM$.
	
Let $SM \coloneqq \{ (x,X): x\in M, X \in T_x M:\  <X,X>_{g(0), x} = 1 \}$ be the $g(0)$-sphere bundle over $M$, whose topology is of course the restriction of the topology on $TM$. Since $M$ is compact, $SM$ is compact as well. Now, since $h_j$ is a continuous section of $ T^{2,0}M$ for all $j\in\{1,\ldots,N\}$, the map $(x,X) \mapsto \inner{X}{X}{h_j,x}$ is a continuous function over $SM$. Indeed, choose any finite atlas $(U_k,\psi_k)_{k=0}^K$ of $M$ consisting of local charts with orthonormal coordinates for $g(0)$, which exists by compactness of $M$. Let $x \in M$. There exists $k \in \{ 0,\ldots, K \}$ such that $x \in U_k$. Write $e_1,\ldots,e_d$ for the local coordinates in the chart $(U_k,\psi_k)$. Now write $h_{j,p,q} : x \in U_k \mapsto \inner{e_p}{e_q}{h_j,x}$, and note this is a continuous function over $U_k$ for all $j\in\{0,\ldots,N\}$ and $p,q\in\{1,\ldots,d\}$. 

Fix $(x,X) \in SM$ and write $X = \sum_{j=1}^d X_j e_j(x)$. We now prove that $(y,Y) \mapsto \inner{Y}{Y}{h_j,y}$ is continuous at $(x,X)$.

Let $\varepsilon > 0$, without loss of generality assume $\varepsilon \leq 1$. By continuity of $h_{j,p,q}$ at $x$, there exists an open subset $V_x$ of $U_k$ such that, for all $y \in V_x$, and for all $p,q \in \{1,\ldots, d\}$, 
\begin{equation*}
	|h_{j,p,q}(y) - h_{j,p,q}(x)| < \frac{\varepsilon}{\inner{X}{X}{g(0),x}+1}\text.
\end{equation*}
Fix $y \in V_x$. We now compute, for any $Y = \sum_{j=1}^d Y_j e_j(y) \in T_y M$ with $\sum_{j=1}^d |X_j-Y_j|^2 < \frac{\varepsilon}{2d^2(1+|h_{j,p,q}(x)|)}$:
\begin{align*}
	|(h_j)_y (Y,Y) - (h_j)_x(X,X)| 
	&\leq \sum_{p=1}^d |X_p^2 h_{j,p,q}(x) - Y_p^2 h_{j,p,q}(y)| \\
	&\leq \sum X_p^2 |h_{j,p,q}(x) - h_{j,p,q}(y)| + \sum_{p=1}^d |X_p^2 - Y_p^2| |h_{j,p,q}(y)| \\
	&\leq \frac{\varepsilon}{2} + d^2 (|h_{j,p,q}(x)|+1)\frac{\varepsilon}{2 d^2|h_{j,p,q}(x)|+1} = \varepsilon \text. 
\end{align*}
Therefore, as claimed, for each $j \in \{0,\ldots,N\}$, the function $(x,X)\in SM\mapsto \inner{X}{X}{h_j,x}$ is continuous over the compact set $SM$. Therefore there exists $C>0$ such that, for all $(x,X) \in SM$, and for all $j\in\{1,\ldots,N\}$, we have $\inner{X}{X}{h_j,x} \leq C$. By homogeneity, we therefore conclude that, for all $(x,X) \in TM$, we have $\inner{X}{X}{h_j,x} \leq \inner{X}{X}{g(0),x}$.

Hence,  for all $(x,X) \in TM$, 
\begin{equation*}
	\left| \frac{d\inner{X}{X}{g(t),x}}{dt} \right| \leq \sum_{j=0}^{N-1} j t^j \inner{X}{X}{h_j,x} \leq C \inner{X}{X}{g(0),x} \sum_{j=0}^{N-1} (j+1)t^j  \text. 
\end{equation*}

Therefore, for all $(x,X) \in TM$, and for all $t,t_0 \in [0,1]$:
\begin{equation*}
	|\inner{X}{X}{g(t),x} - \inner{X}{X}{g(t_0),x}| \leq C |\inner{X}{X}{g(0),x} |t-t_0| \text,
\end{equation*}
as claimed.
\end{proof}
 
\begin{notation}
	If $g$ is a Riemannian metric on the  closed connected Riemannian manifold $M$, then we will denote by $d_g$ the geodesic distance induced by $g$ on $M$, and by $\Lip_g$ the associated Lipschitz seminorm on $C(M)$.
\end{notation}

\begin{lemma}\label{Lipschitz-implies-Lipschitz-lemma}
	Let $M$ be a closed connected Riemannian manifold. Assume  that  $\{g(t)\}_{t\in [0,1]}$ is a family of Riemannian metrics on $M$ with  the following property:  for all $(x,X) \in TM$ and for all $t\in [0,1]$,
	\begin{equation*}
		|\inner{X}{X}{g(t),x} - \inner{X}{X}{g(0),x}|\leq C t \inner{X}{X}{g(0),x}\text,
	\end{equation*}
	for some constant $C>0$.
	Then for all $t\in [0,1]$,
	\begin{equation*}
		\dom{\Lip_{g(0)}} = \dom{\Lip_{g(t)}}\quad  \text{ and }\quad \frac{1}{C t +1}\Lip_{g(0)}\leq \Lip_{g(t)} \leq (C t +1)\Lip_{g(0)}.
	\end{equation*}
\end{lemma}

\begin{proof}
Fix $x,x' \in M$. Let $\gamma$ be a $C^1$ path from $x$ to $x'$ in $M$. By definition of the geodesic distance for a Riemannian manifold,
\begin{align}\label{Lipschitz-implies-Lipschitz-eq1}
  d_{g(0)}(x,x') 
  &\leq \int_0^1 \sqrt{\inner{\frac{d\gamma}{ds}}{\frac{d\gamma}{ds}}{g(0),\gamma(s)}} \; ds \\
  &\leq (C t+1) \int_0^1 \sqrt{\inner{\frac{d\gamma}{ds}}{\frac{d\gamma}{ds}}{g(t), \gamma(s)}} \; ds \text.
\end{align}
As $\gamma$ above is an abitrary path in $M$ from $x$ to $x'$, we conclude from Expression (\ref{Lipschitz-implies-Lipschitz-eq1}) that $(C t+1)^{-1} d_{g(0)}(x,x') \leq d_{g(t)}(x,x')$.

A similar computation shows that
\begin{align}
  d_{g(t)}(x,x') 
  &\leq \int_0^1 \sqrt{\inner{\frac{d\gamma}{ds}}{\frac{d\gamma}{ds}}{g(t),\gamma(s)}} \; ds\\
  &\leq (C t +1) \int_0^1 \sqrt{\inner{\frac{d\gamma}{ds}}{\frac{d\gamma}{ds}}{g(0),\gamma(s)}} \; ds \text,
\end{align}
and again taking the infimum over all path $\gamma$ from $x$ to $x'$, we get $d_{g(t)}\leq (C t +1) d_{g(0)}$.

By definition, for any $f \in C(M)$ and allowing for $\infty$, it then follows that
\begin{equation*}
	\frac{1}{(C t +1)}  \Lip_{g(0)} \leq \Lip_{g(t)} \leq (C t +1)\Lip_{g(0)} \text,
\end{equation*}
as claimed.
\end{proof}

\begin{definition}\label{def:LipD}
	Denote by $\dil{.}$ the dilation \cite{Latremoliere16b}.  The \emph{Lipschitz distance} $\LipschitzD((\A,\Lip_\A),(\B,\Lip_\B))$ between two {\qcms s} $(\A,\Lip_\A)$ and $(\B,\Lip_\B)$ is defined by:
	
	\begin{multline*}
		\LipschitzD((\A,\Lip_\A),(\B,\Lip_\B))\coloneqq
		\inf\{ \max\{\ln\dil{\pi},\ln\dil{\pi^{-1}}\} : \\ \pi:(\A,\Lip_\A)\rightarrow(\B,\Lip_\B) \text{ bi-Lipschitz isomorphism } \} \text.
	\end{multline*}
	with the convention that $\inf\emptyset=\infty$.
\end{definition}

\begin{corollary}\label{Lipschitz-cv-cor}
	Let $M$ be a closed connected Riemannian  manifold. 
	Assume  that  $\{g(t)\}_{t\in [0,1]}$ is a family of Riemannian metrics on $M$ with  the following property:   for all $x\in M$, for all $X,Y \in   T_xM$, 
and  for all $t\in [0,1]$,
	\begin{equation*}
		|\inner{X}{Y}{g(t), x} - \inner{X}{Y}{g(0),x}|\leq C \inner{X}{Y}{g(0)}(x) |t|
	\end{equation*}
	for some $C>0$. Then    
	\begin{equation*}
		\lim_{t\rightarrow 0}\LipschitzD((C(M),\Lip_{g(t)}),(C(M),\Lip_{g(0)})) = 0 \text.
	\end{equation*}
\end{corollary}

\begin{proof}
	By Lemma (\ref{Lipschitz-implies-Lipschitz-lemma}), the identity of $C(M)$ is a Lipschitz isomorphism from $(C(M),\Lip_{g(t)})$ to $(C(M),\Lip_{g(0)})$ with  Lipschitz constant $C t+1$. Its  inverse's Lipschitz constant is also $C t+1$. Our conclusion then follows from Definition \eqref{def:LipD} and the fact that $\lim_{t\rightarrow 0} \ln(C t+1) = \ln(1)=0$.
\end{proof}

\begin{theorem}[{\cite[Lemma 4.6]{Latremoliere16b}}]\label{thm:Lip-qcms}
	If $(\A,\Lip_\A)$ and $(\B,\Lip_\B)$ are {\qcms s} with $\LipschitzD((\A,\Lip_\A),(\B,\Lip_\B)) < \infty$, then:
	\begin{multline*}
		\dpropinquity{}((\A,\Lip_\A),(\B,\Lip_\B)) \leq \\ \exp(\LipschitzD((\A,\Lip_\A),(\B,\Lip_\B))-1) \max\{ \qdiam{\A}{\Lip_\A},\qdiam{\B}{\Lip_\B} \} \text.
	\end{multline*}
	In particular, Lipschitz convergence implies convergence for the propinquity.
\end{theorem}
Note that the diameter of a {\qcms} is continuous with respect to the Lipschitz convergence (and the propinquity), so it is bounded for a convergent family.

\begin{remark}\label{base-tunnel-remark}
	Theorem \eqref{thm:Lip-qcms} follows from \cite[Proposition 3.80]{Latremoliere15}, where a tunnel is constructed which, in our case, will be of the form:
	\begin{equation*}
		\left(C(M)\oplus C(M), \underbracket[1pt]{\left\{\max\{\Lip_{g(t)},\Lip_{g(0)}, (f,g) \mapsto \frac{1}{K(t)} \norm{f-g}{C(M)}\right\}}_{\text{tunnel L-seminorm}}, \underbracket[1pt]{f\oplus g\mapsto f, f\oplus g\mapsto g}_{\text{canonical surjections}}\right)
	\end{equation*}
	with $K(t) \coloneqq C|t|\diam{M}{}$.
\end{remark}

\begin{remark}
	The first assumption in Theorem (\ref{main-thm}) is chosen to make the modular Leibniz property hold; the rest of our argument only relies on the properties assumed on the Dirac operators.
\end{remark}

We are now ready to prove our second main result.

\begin{proof}[Proof of Theorem (\ref{Riemannian-cv-thm})] 
	
		Let $t\in I\coloneqq [0,1] \mapsto g(t)$ be a polynomial path of $C^\infty$ Riemannian metrics over $M$. For each $t\in I$, let $\Gamma^2 \mathrm{Spin}_{g(t)}$ be the Hilbert space of square integrable sections of the spinor bundle over $M$ for the metric $g(t)$, and $\Dirac^t$ the associated Dirac operator. We also denote $\Gamma^2 \mathrm{Spin}_{g(0)}$ by $\Hilbert$, and by $\Dirac_0$ by  $\Dirac$. 
	
	Since polynomial paths of $C^\infty$-Riemannian metrics are, in particular, analytic paths of metrics, by \cite{Bo-Gau, NowaczykThesis, Maier, Hermann12}, there exists a family of unitaries $t\in [0,1] \mapsto \beta(t)$ with $\beta(t):\Gamma^2 \mathrm{Spin}_{g(t)} \rightarrow \Hilbert$, such that:
	\begin{itemize}
		\item $\beta(t)$ is a unitary from $\Gamma^2 \mathrm{Spin}_{g(t)}$ onto $\Hilbert$, which intertwines the action of $C(M)$ on $\Gamma^2 \mathrm{Spin}_{g(t)}$ and $\Hilbert$ (note that we will omit writing a special symbol for these representations),
		\item If we set, $\Dirac_t \coloneqq \beta(t)\Dirac^t \beta(t)^\ast$, for all $t\in [0,1]$, then $t\in [0,1] \mapsto \Dirac_t$ is a holomorphic family of self-adjoint operators of type (A) \cite[Section VII \S 2]{Kato}.
	\end{itemize}

Moreover,  by \cite[Section VII \S 5, Theorem 3.9]{Kato},  there exist a sequence $(\lambda_n)_{n\in\N}$ of continuous real-valued functions with domain  $[0,1]$, and a sequence $(e_n)_{n\in\N}$ of continuous functions from   $[0,1]$ to $\Hilbert$, such that for all $t\in[0,1]$ and $n\in\N$, we have $\Dirac_t e_n(t) = \lambda_n(t) e_n(t)$. Moreover  $(e_n(t))_{n\in\N}$ is an orthonormal basis of $\Hilbert$.

In addition, by  \cite[Theorem 2.2]{Nowaczyk2013}, for all $\Lambda>0$, there exists $N\in\N$ and $\delta>0$ such that, if $t\in[0,\delta)$ then $|\{\lambda_n(t) : n\in\N\}| = N$.

Fix $t \in [0,1]$. Since 
	\begin{multline*}
		\Lip(t) : a \in \sa{\A} \mapsto \opnorm{[\Dirac^t,a]}{}{\Gamma^2 \mathrm{Spin}_{g(t)}} = \opnorm{[\beta(t)^\ast \Dirac_t \beta(t), a]}{}{\Gamma^2 \mathrm{Spin}_{g(t)}} \\ = \opnorm{\beta(t)^\ast  [\Dirac_t,a] \beta(t)}{}{\Gamma^2 \mathrm{Spin}_{g(t)}} = \opnorm{[\Dirac_t,a]}{}{\Hilbert}\text,
	\end{multline*} 
	the map $(\mathrm{Ad}_{\beta(t)},\beta(t))$ is an isometry between spectral triples, and thus the spectral propinquity between $(C(M),\Gamma^2 \mathrm{Spin}_{g(t)},\Dirac^t)$ and $(C(M),\Hilbert,\Dirac_t)$ is $0$. 
	Therefore to show our claim is enough to prove that $\lim_{t\rightarrow 0} \spectralpropinquity{}((C(M),\Hilbert,\Dirac_0),(C(M),\Hilbert,\Dirac_t)) = 0$.
	
	By Corollary (\ref{Lipschitz-cv-cor}), $(C(M),\Lip(t))_{t\in I}$ converges to $(C(M),\Lip(0))$ for the Lipschitz distance. Our conclusion then follows directly from Theorem (\ref{main-thm}).
\end{proof}

%%%%%%%%%

\vfill

\begin{thebibliography}{10}
	
	
\bibitem{Aguilar15}	K. Aguilar and F. Latr\'emoli\`ere, \emph{Quantum ultrametrics on AF algebras and the Gromov-Hausdorff propinquity},  Studia Math. \textbf{231} (2015), 149--193.  
	
	
\bibitem{Ammann16}	B. Ammann, H. Weiss, and F.  Witt, 
\emph{A spinorial energy functional: critical points and gradient flow}, 
	Math. Ann. \textbf{365} (2016), 1559--1602.
	
\bibitem{Bandara18} L. Bandara,  A. McIntosh, and A.  Ros\'en, 
\emph{Riesz continuity of the Atiyah-Singer Dirac operator under perturbations of the metric,}
Math. Ann. \textbf{370}  (2018),  863--915. 

	
\bibitem{Bandara19}	L. Bandara,  A. Ros\'en,
	\emph{Riesz continuity of the Atiyah-Singer Dirac operator under perturbations of local boundary conditions}, 
	Comm. Partial Differential Equations  \textbf{44}(2019),  1253--1284.
	
		\bibitem{Bando-Urakawa}	S. Bando and H.  Urakawa,
	{\it	Generic properties of the eigenvalue of the Laplacian for compact Riemannian manifolds},
	Tohoku Math. J. {\bf 35} (1983), 155--172.
		

	\bibitem{Bar-96}	C. B\"ar. \emph{Metrics with harmonic spinors},  Geom. Funct. Anal. \textbf{6.6} (1996),pp. 899--942.
	
\bibitem{Bar-05} 	C. B\"ar, P. Gauduchon, and  A. Moroianu,
\emph{Generalized cylinders in semi-Riemannian and Spin geometry}, 
	Math. Z. \textbf{249} (2005),  545--580.
	
	


	
%	\bibitem{Latremoliere23b}
%	{J}.~{B}assi, {R}.~{C}onti, {C}.~{F}arsi, and {F}. {L}atr{\'e}moli{\`e}re,
%	\emph{Spectral Gromov--Hausdorff limits and isometries,} J. London Math. Soc \textbf{108},  (2023), 1488--1530,
%	arXiv:2302.09117.
	
		\bibitem{Ben-Artzi16} J. Ben-Artzi and  T.  Holding, 
	\emph{Approximations of strongly continuous families of unbounded self-adjoint operator}, Commun. Math. Phys. \textbf{345} (2016), 615--630 .
	
		\bibitem{Berger73} M. Berger, \emph{Sur les premi\`eres valeurs propres des vari\'et\'es riemanniennes}, Compositio Math. \textbf{26}  (1973), 129--14
	
	
	\bibitem{Binz83} E. B\'inz and R.  Pferschy, \emph{The Dirac operator and the change of the metric},  C.R. Math. Rep.
	Acad. Sci. Canada \textbf{5} (1983) 269--274.
	
	
	\bibitem{Bo-Gau} J.-P. Bourguignon and P. Gauduchon. \emph{Spineurs, operateurs de Dirac et variations de metriques},  Comm. Math. Phys.   \textbf{144} (1992), pp. 581--599.
	
	
%	\bibitem{Rieffel15b}
%	{M}. {C}hrist and {M.}~{A.} {R}ieffel, \emph{Nilpotent group
%		{$C^\ast$-algebras}-algebras as compact quantum metric spaces}, Canad. Math.
%	Bull. \textbf{60} (2017), no.~1, 77--94, arXiv: 1508.00980.
	
		\bibitem{Canzani14}
	Y. Canzani, \emph{On the multiplicity of eigenvalues of conformally covariant operators}, 	Ann. Inst. Fourier Grenoble \textbf{64} (2014), 947--970.
	
	
%	\bibitem{Connes89}
%	A.~{C}onnes, \emph{Compact metric spaces, {F}redholm modules and
%		hyperfiniteness}, Ergodic Theory Dynam. Systems \textbf{9} (1989), no.~2,
%	207--220.
%	
%	\bibitem{Connes}
%	\bysame, \emph{Noncommutative geometry}, Academic Press, San Diego, 1994.
%	
	
		\bibitem{Dabrow13}	L. Dąbrowski G. Dossena, 
	\emph{Dirac operator on spinors and diffeomorphisms}, 
	Classical Quantum Gravity \textbf{30} (2013),  015006, 11 pp.
	
	\bibitem{Dabrow86}	L. Dąbrowski and R.  Percacci, 
	Spinors and diffeomorphisms.
	Comm. Math. Phys. \textbf{106} (1986),  691--704.
	

	
	
%	\bibitem{Latremoliere23a}
%	{C}.~{F}arsi, {J}.~{P}acker, and {F}.~{L}atr{\'e}moli{\`e}re \emph{Convergence of
%		inductive sequences of spectral triples for the spectral propinquity}, Adv. Math, Paper No. 109442, 59 pp., arXiv:2301.00274.
		
		
	

%	\bibitem{Fukaya87}
%	{K}. {F}ukaya, \emph{{C}ollapsing of {R}iemannian manifolds and eigenvalues of
%		{L}aplace operator}, Invent. Math. \textbf{87} (1987), no.~3, 517--547.
%	
%	\bibitem{Gromov81}
%	M.~{G}romov, \emph{Groups of polynomial growth and expanding maps}, Publ. Math.
%	Inst. Hautes \'Etudes Sci. \textbf{53} (1981), 53–78.
%	
%	\bibitem{Gromov}
%	\bysame, \emph{Metric structures for {R}iemannian and non-{R}iemannian spaces},
%	Progress in Mathematics, Birkh\"auser, 1999.

	\bibitem{Hermann12}	A. Hermann, \emph{Dirac eigenspinors for generic metrics}, 
	Dissertation zur Erlangung des Doktorgrades
	der Naturwissenschaften, Universit\"at Regensburg vorgelegt von Regensburg, im Januar 2012, arXiv:1201.5771v3.
	
	
\bibitem{Kaad21} J. 	Kaad, and D. Kyed, 
\emph{Dynamics of compact quantum metric spaces},
	Ergodic Theory Dynam. Systems \textbf{41} (2021),  2069--2109.
	
		\bibitem{Karpukhin24}	M. Karpukhin,A.  M\'etras, and I. Polterovich,
\emph{Dirac eigenvalue optimisation and harmonic maps to complex projective spaces}, 
	Int. Math. Res. Not. IMRN 2024, \textbf{21}, 13758--13784.
	
		\bibitem{Kato}	T. Kato,
	{\it Perturbation theory for linear operators},
	Second edition. Grundlehren der Mathematischen Wissenschaften, Band 132. Springer-Verlag, Berlin-New York, 1976. xxi+619 pp. 
	
	
	
	
	\bibitem{Kriegl03}		A. Kriegl and P.  W. Michor
	\emph{Differential perturbation of unbounded operators},
	 Math. Ann. \textbf{327}, 1 (2003), 191--201.

	

	
%	\bibitem{Latremoliere05}
%	\bysame,  \emph{Approximation of the quantum tori by finite
%		quantum tori for the quantum {G}romov-{H}ausdorff distance}, J. Funct. Anal.
%	\textbf{223} (2005), 365–395, math.OA/0310214.
	
%	\bibitem{Latremoliere13c}
%	\bysame, \emph{Convergence of fuzzy tori and quantum tori for the quantum
%		{G}romov--{H}ausdorff propinquity: an explicit approach.}, M\"nster J. Math.
%	\textbf{8} (2015), no.~1, 57–98, arXiv: math/1312.0069.
%	
%	\bibitem{Latremoliere15c}
%	\bysame, \emph{Curved noncommutative tori as {L}eibniz compact quantum metric
%		spaces}, J. Math. Phys. \textbf{56} (2015), no.~12, 123503, 16 pages, arXiv:
%	1507.08771.
%	



	\bibitem{Latremoliere13b}
	F. Latr\'emoli\`ere, \emph{The dual {G}romov--{H}ausdorff propinquity}, J. Math. Pures
	Appl. \textbf{103} (2015), no.~2, 303--351, arXiv: 1311.0104.
	
	\bibitem{Latremoliere15}
	\bysame, \emph{A compactness theorem for the dual {G}romov-{H}ausdorff
		propinquity}, Indiana Univ. Math. J. \textbf{66} (2017),  1707--1753,
	arXiv: 1501.06121.
	
	\bibitem{Latremoliere16b}
	\bysame, \emph{Equivalence of quantum metrics with a common domain}, J. Math.
	Anal. Appl. \textbf{443} (2016), 1179–1195, arXiv: 1604.00755.
	

	
	\bibitem{Latremoliere16c}
	\bysame, \emph{The modular {G}romov--{H}ausdorff propinquity}, Dissertationes
	Math. \textbf{544} (2019), 1--70, arXiv: 1608.04881.
	
	
	
	\bibitem{Latremoliere18b}
	\bysame, \emph{The covariant {G}romov--{H}ausdorff propinquity}, Studia Math.
	\textbf{251} (2020), no.~2, 135--169, arXiv: 1805.11229.
	
		\bibitem{Latremoliere18c}
	\bysame, \emph{Convergence of {C}auchy sequences for the covariant
		{G}romov--{H}ausdorff propinquity}, J. Math. Anal. Appl. \textbf{469} (2019),
	no.~1, 378--404, arXiv: 1806.04721.
	
%	\bibitem{Latremoliere21a}
%	\bysame, \emph{Convergence of spectral triples on fuzzy tori to spectral
%		triples on quantum tori}, Comm. Math. Phys. \textbf{388} (2021), no.~2,
%	1049--1128, arXiv: 2102.03729.
	
	\bibitem{Latremoliere18d}
	\bysame, \emph{The dual-modular {G}romov--{H}ausdorff propinquity and	completeness}, {J}. {N}oncomm. {G}eometry \textbf{115} (2021), no.~1,
	347--398.
	
	\bibitem{Latremoliere18g}
	\bysame, \emph{The {G}romov--{H}ausdorff propinquity for metric spectral
		triples}, Adv. Math. \textbf{404} (2022), Paper No. 108393, 56.
		
			\bibitem{Latremoliere22}
		\bysame,  \emph{Continuity of the spectrum of Dirac	operators of spectral triples for the spectral propinquity}, Math. Ann., Pub Date: 2023--07--06 , DOI:10.1007/s00208-023-02659-x. 49 pages.
		
		\bibitem{Lesch05}	M. Lesch, \emph{The uniqueness of the spectral 
			flow on spaces of unbounded self-adjoint Fredholm operators}, Spectral geometry of manifolds with boundary and decomposition of manifolds,
		Contemp. Math., vol. 366, Amer. Math. Soc., Providence, RI, 2005, pp. 193--224.
	
	\bibitem{Li05}	H. Li,
	\emph{$\theta$-deformations as compact quantum metric spaces}, 
	Comm. Math. Phys. \textbf{256} (2005), 213–238.

\bibitem{Maier}  S. Maier, 
	{\it Generic metrics and connections on Spin- and and $Spin^c$-manifolds},
	Comm. Math. Phys. {\bf 188} (1997),  407--437. 
	
\bibitem{Muller18}  E.-V.	M\"uller, {\it An obstruction to higher--dimensional kernels of Dirac operators},  arXiv:1802.06568v1 .

\bibitem{Muller17} 	O. M\"uller and N. Nowaczyk, 
\emph{A universal spinor bundle and the Einstein-Dirac-Maxwell equation as a variational theory}, 
	Lett. Math. Phys. \textbf{107} (2017), no. 5, 933--961.
	
	

\bibitem{Nowaczyk2013} N. Nowaczyk, {\it Continuity of Dirac spectra},  Ann. Global Anal. Geom. {\bf 44} (2013), 541--563.

\bibitem{NowaczykThesis} N. Nowaczyk, {\it 	Nowaczyk, N.: Dirac eigenvalues of higher multiplicity. Ph.D. thesis},  14 Jan 2015


\bibitem{Nowaczyk2016} N. Nowaczyk, {\it Existence of Dirac eigenvalues
of higher multiplicity,} Math. Z. \textbf{284  } (2016) 285--307.

\bibitem{Perez08}	R.F.  P\'erez,  and J. M. Masqu\'e,
\emph{Natural connections on the bundle of Riemannian metrics}, 
Monatsh. Math. \textbf{155} (2008), 67--78.


\end{thebibliography}
\end{document}